\documentclass[11pt]{article}

\usepackage{amsmath,amsthm,amssymb,color}
\usepackage{graphicx}
\usepackage{tikz-cd}
\usepackage{amscd}

\parskip=\medskipamount
\def\clift#1{#1^{\scriptscriptstyle{\mathrm{C}}}}

\def\vlift#1{#1^{\scriptscriptstyle{\mathrm{V}}}}

\def\fpd#1#2{{\displaystyle\frac{\partial #1}{\partial #2}}}
\def\spd#1#2#3{{\displaystyle\frac{\partial^2 #1}
{\partial #2\partial #3}}}

\def\clift#1{#1^{\scriptscriptstyle{\mathrm{C}}}}

\def\vlift#1{#1^{\scriptscriptstyle{\mathrm{V}}}}
\def\vilms#1{#1^{\scriptscriptstyle{\mathrm{Vilms}}}}

\def\R{\mathbb{R}}

\font\frak=eufm10 scaled\magstep1
\def\goth#1{\hbox{{\frak #1}}}
\def\g{\goth{g}}
\def\la{\g}

\def\cinfty#1{C^{\scriptscriptstyle\infty}(#1)}
\def\vectorfields#1{{\mathcal X}(#1)}

\def\Ad{\mathop{\mathrm{ad}}\nolimits}
\def\ad{\Ad}

\def\sode{{\sc sode}}

\def\T{{\mathbf T}}

\def\v{{\mathsf v}}
\def\w{{\mathsf w}}

\theoremstyle{plain}

\newtheorem{proposition}{Proposition}
\newtheorem{definition}{Definition}

\begin{document}

\title{Nonlinear splittings  on fibre bundles}

\author{ S.\ Hajd\'u$^\dagger$ and T.\ Mestdag$^{\dagger\,\ddagger}$\footnote{Corresponding author} \\[2mm]
	{\small $^\dagger$ Department of Mathematics,  University of Antwerp,}\\
	{\small Middelheimlaan 1, 2020 Antwerpen, Belgium}\\[2mm]
	{\small $^\ddagger$ 
		Department of Mathematics: Analysis, Logic and Discrete Mathematics, Ghent  University}\\
	{\small Krijgslaan 281, 9000 Gent, Belgium}\\[2mm]
	{\small Email: sandor.hajdu@uantwerpen.be, tom.mestdag@uantwerpen.be} 
}

\date{}

\maketitle

\begin{abstract}
We introduce the notion of a nonlinear splitting on a fibre bundle as a generalization of  an Ehresmann connection. We present its basic properties and we pay  attention to the special cases of affine, homogeneous and  principal nonlinear splittings. We explain where nonlinear splittings appear in the context of Lagrangian systems and Finsler geometry and we show their relation to Routh symmetry reduction,  submersive second-order differential equations and unreduction. We define a curvature map for a nonlinear splitting, and we indicate where this concept appears in the context of nonholonomic systems with affine constraints and Lagrangian systems of magnetic type.

\vspace{3mm}


\vspace{3mm}


\end{abstract}

\section{Introduction}

Principal, linear and nonlinear connections on principal, vector, frame and fibre bundles are among the most indispensable tools of  differential geometry. Moreover, they  have been applied to formulate and solve many problems in dynamical systems and mathematical physics. One may think e.g.\ of the Levi-Civita connection in Riemannian geometry, of the many distinct linear and nonlinear connections in Finsler geometry \cite{Szilasi},  of the  principal connections that describe linear nonholonomic constraints  in geometric mechanics \cite{Bloch} or of the connections that appear in Lagrangian field theories and  gauge theories \cite{Sardan}. 

In this paper we will use the terminology `Ehresmann connection' for a `standard' connection on a fibre bundle $\pi: M\to N$. With that we mean a direct complement $H\pi$ of the vertical distribution $V\pi={\rm Ker}\,T\pi$ within $TM$. It is well-known (see e.g.\ \cite{natural}) that this notion can be equivalently cast in terms of a horizontal lift, which is in essence a splitting of a certain short exact sequence of vector bundles. The goal of this paper is to show that the notion of an Ehresmann connection can be meaningfully extended, if we drop the requirement that the splitting is a linear bundle map. This will give rise to the concept that we have termed `a nonlinear splitting' in this paper. In contrast with an Ehresmann connection, the image of a non-linear splitting is a submanifold of $TM$, and no longer a distribution on $M$. 

A nonlinear splitting should not be confused with the concept that is often called a `nonlinear connection'. A nonlinear connection often refers to a  (standard) Ehresmann connection, in case the fibre bundle is a vector bundle, and the adjective `nonlinear' is usually added  to distinguish Ehresmann connections from linear connections on that vector bundle. For example, the nonlinear connection that can be associated to a \sode\ (see e.g.\ \cite{Szilasi}) is in fact an Ehresmann connection on the vector  bundle $\tau:TM \to M$.

After some preliminaries, we investigate in Section~\ref{sec2} both the similarities and the differences between nonlinear splittings and Ehresmann connections, at the level of their horizontal projections and Vilms lifts. 
As an application, we indicate how nonlinear splittings appear in the context of nonholonomic systems. With an eye on future applications in Finsler geometry, we introduce in Section~\ref{sec4}  the notion of a homogeneous nonlinear splitting, and we prove necessary and sufficient conditions for a nonlinear splitting to be either homogeneous, or an Ehresmann connection.

In Section~\ref{subduced}  we consider Lagrangian systems on $M$, where $M$ is the total manifold of a fibre bundle over $N$, and we show that under the appropriate condition of fibre-regularity one may associate a nonlinear splitting to this Lagrangian system. Under a further symmetry-type condition it can be shown that `horizontal' solutions of the Lagrangian system on $M$ are in fact related to the solutions of a  Lagrangian system on the base manifold $N$ (see Proposition~\ref{symmetryprop}).   The corresponding Lagrangian on $N$ is what we call `the subduced Lagrangian', following \cite{popescu1,popescu2}. We end the section with a discussion on the relation of the proposition to submersive systems of second-order ordinary differential equations \cite{kossowski,willysubmersive}.

In the special case that the fibre bundle is a principal bundle, we give a necessary and sufficient condition for a principal nonlinear splitting to be a principal connection. We show in Section~\ref{secprincipal} how the results of the previous section fit within the context of  reduction of a Lagrangian system with a symmetry Lie group. We discuss some aspects of Lagrange-Poincar\'e reduction \cite{CMR,Inv}, Routh reduction \cite{Bavo,tomsrouth} and unreduction \cite{unreduction}.

In most of the applications where Ehresmann connections are being used, its curvature plays an important role. We give a definition for the curvature of a nonlinear splitting. To demonstrate the significance of this definition to future applications, we show in Section~\ref{seccurvature} where the curvature of an affine nonlinear splitting appears in the geometric modeling of a Lagrangian system with extra magnetic forces and of mechanical systems with affine nonholonomic constraints.   
The paper ends with an outlook to an application of nonlinear splittings in Finsler geometry.

\section{Nonlinear Splittings} \label{sec2}

	Let $\pi: M\rightarrow N$ be a fibre bundle. Throughout we will use $\tau:TM\rightarrow M$ and $\bar\tau:TN\rightarrow N$ for the tangent bundles of $M$ and $N$, respectively. We will consider the pullback of $\bar\tau$ by $\pi$:
	\[
	\pi^*TN=\{(m,v_n)\in M\times TN~|~\pi(m)=\bar\tau(v_n)\},
	\]
and use $p_1$ and $p_2$ for the projections $\pi^*TN \to M$ and $\pi^*TN \to TN$, respectively.	Sections of this pullback bundle  can also be thought of as maps $\eta: M \to TN$ with $\bar\tau \circ \eta = \pi$. In what follows we will call such a section `a vector field on $N$ along $\pi$' and we will denote the set of these sections by $\vectorfields{\pi}$.  

Let $\mu: TM \to \pi^*TN$ be the linear bundle map $(\tau,T\pi)$, i.e.\ $\mu(w_m)= (m, T\pi(w_m))$. 	
	With these ingredients we can write down the short exact sequence	\[
	\begin{tikzcd}
	0\arrow{r} & V\pi\arrow{r}{} & TM\arrow{r}{\mu} & \pi^*TN\arrow{r} & 0,
	\end{tikzcd}
	\]
	where $V\pi$ stands for the vertical bundle ${\rm Ker}\, T\pi$ of $\pi$. Each of the manifolds in this sequence is fibred over $M$. In the following definition, we consider a right splitting of the sequence, but, importantly, we do not assume it to be linear in the fibre coordinates.
	
\begin{definition}
A nonlinear splitting on $\pi:M\rightarrow N$
	 is a map $h:\pi^*TN\rightarrow TM$ which is
	\begin{itemize}
		\item smooth on the slit pullback bundle $\pi^*\mathring{T}N$,
		\item fibre-preserving, i.e.\ $\tau \circ h = p_1$,
		\item satisfies $T \pi \circ h = p_2$.
	\end{itemize}
	We call ${\mathcal H} = {\rm Im}\,h\subset TM$ the horizontal manifold of $h$.
\end{definition}

In the definition, $\mathring{T}N$ stands for the tangent manifold $TN$ from which the zero section has been removed. This subtle aspect will become important when we consider homogeneous nonlinear splittings in Section~\ref{sec4}.

In what follows we will often use coordinates $(x^i)$ on $N$ and  coordinates $(q^a)=(x^i,y^{\alpha})$ on $M$ that are adjusted to the fibre bundle structure of $\pi$. We will denote the corresponding natural fibre coordinates on $TM$ by $(u^a)=(v^i,w^{\alpha})$. Locally, a nonlinear splitting $h$ can then be expressed as
\[
h(x^i,y^{\alpha},v^i)=(x^i,y^{\alpha},v^i,h^{\alpha}(x,y,v)).
\]
We will refer to the functions $h^{\alpha}$ as the `coefficients of $h$'.

The coordinates $v^i$ also represent the fibre coordinates of the vector bundle $\bar\tau: TN \to N$. It is clear from the coordinate expression that if we would require $h$ to be a linear map between the two vector bundles $\pi^*TN$ and $TM$, the  coefficients $h^\alpha$ would be linear in the $v^i$-coordinates. In that case, we would obtain an Ehresmann connection on a fibre bundle (see e.g. \cite{natural}). We will explore this in more detail in the next section, but first we show that much of the apparatus of Ehresmann connections can be  transferred to the current (more general) setting. 

 The {\em horizontal projection operator} of a nonlinear splitting $h:\pi^*TN\rightarrow TM$ is the map $P_h:TM\rightarrow TM, w_m \mapsto h(m,T\pi(w_m))$. In local coordinates, we obtain the expression
\[
P_h(x^i,y^{\alpha},v^i,w^{\alpha})=(x^i,y^{\alpha},v^i,h^{\alpha}(x,y,v)).
\]
The {\em vertical projection operator} is the map $P_v:TM\rightarrow TM$, given by $P_v(w_m) = w_m- h(m,T\pi(w_m))$. It is clear that $P_v(w_m) \in V\pi$  since $T\pi(P_v(w_m)) =T\pi(w_m) -T\pi(w_m)=0$. $P_v$ is the left (nonlinear) splitting of the short exact sequence. It is easy to see that
\[
P_v(x^i,y^{\alpha},v^i,w^{\alpha})=(x^i,y^{\alpha},0,w^{\alpha}-h^{\alpha}(x,y,v)).
\]
With these operators, we can decompose $w_m=P_h(w_m)+P_v(w_m)$ and we have the properties $P_h \circ P_h = P_h$ and $P_v\circ P_h=0$. However,
\[
P_v( P_v(w_m))=P_v(w_m)+P_v(0_m)  \qquad \mbox{and} \qquad P_h( P_v(w_m))= P_h(0_m).
\]
Herein is $P_h(0_m) = -P_v(0_m) = (x^i,y^\alpha,0,h^\alpha(x,y,0))$. Also, notice that the horizontal and vertical projection operators of a nonlinear splitting can not be thought of as (1,1)-tensor fields on $M$, since $h$ is not a fibrewise linear mapping. 

Recall the map  $\mu:TM\to \pi^*{TN}$ that appears in the short exact sequence. Tangent vectors (as elements of $TTM$) belong to ${\rm Ker}\, T\mu$ if they are of the type $W^\alpha\partial/\partial {w^\alpha}\mid_{w_m}$. They can also be interpreted as the $\tau$-vertical lift (see Section~\ref{sec4} for its definition) of the $\pi$-vertical vector $W^\alpha\partial/\partial {y^\alpha}\mid_{m}$ to $w_m$. 

\begin{proposition}
A map $P_h: TM \to TM$ is the horizontal projection operator of a nonlinear splitting if and only if  $\tau\circ P_h = \tau$, $T\pi\circ P_h=T\pi$ and ${\rm Ker}\, T\mu \subset {\rm Ker}\,TP_h$.
\end{proposition}
\begin{proof} 
From the first two properties in the statement we see that $P_h$ must be of the type $P_h: (x^i,y^\alpha,v^i,w^\alpha) \mapsto (x^i,y^\alpha,v^i,h^\alpha(x,y,v,w))$. The last property means that $TP_h\mid_{{\rm Ker}\, T\mu}=0$ or $\partial h^\alpha/\partial w^\beta=0$. 
\end{proof}

The map $h$  can be used to lift sections $\eta$ of the pullback bundle $\pi^*TN$ to vector fields $\eta^h$ on $M$. In particular, for a vector field $X$ on $N$ (thought of as a `basic' section of  $\pi^*TN$) its {\em horizontal lift} satisfies
\[
X^h(m) = h(m,X(\pi(m)).
\]
We may also define  the horizontal lift of curves from $N$ to $M$. 

\begin{definition}
Let $h$ be a nonlinear splitting  on $\pi:M\rightarrow N$ and let $c_n(t)$ be a curve on $N$, with $c_n(0)=n$. Its horizontal lift $c^h_m$ to $m\in M_n$ is the unique curve in $M$ such that $\pi\circ c^h_m = c_n$ and ${\dot c}^h_m \in {\mathcal H}$.
\end{definition} 

When $c_n(t)$ is  locally given by $(x^i(t))$, the above horizontal lift is the curve $c^h_m(t)=(x^i(t),y^{\alpha}(t))$ that is determined by the first-order initial value problem
\begin{eqnarray}
\nonumber \dot{y}^{\alpha}(t)&=&h^{\alpha}(x(t),y(t),\dot{x}(t)), \\
\nonumber {y}^{\alpha}(0)&=&y_0^{\alpha},
\end{eqnarray}	
where $(y_0^{\alpha})$ are the fibre coordinates of $m$.

We next show that a construction known as the `Vilms lift of an (Ehresmann) connection' (see e.g.\ \cite{Vilms, unreduction}) can be extended to the current context of nonlinear splittings.  Below, $\sigma: TTM \to TTM$ stands for the canonical involution (see e.g.\ \cite{deleonrodrigues} or \cite{natural}, where it is called the `canonical flip'). In a notation where induced coordinates on $TTM$ are denoted by couples we may write $\sigma: (q^a,u^a,Q^a,U^a) \mapsto (q^a,Q^a,u^a,U^a)$, where, in comparison with earlier notation, $(q^a)=(x^i,y^\alpha)$, $(u^a) = (v^i,w^\alpha)$, etc.

\begin{definition}\label{nonlinearsplitting}
	Let $h$ be a nonlinear splitting on $\pi: M\rightarrow N$ with vertical projection operator $P_v:TM\rightarrow TM$. The unique nonlinear splitting $\vilms{h}$ on $T\pi: TM\rightarrow TN$ whose vertical projection operator is
	\[
	\vilms{P}_v=\sigma\circ TP_v \circ \sigma
	\]
	is called the Vilms lift of $h$.
\end{definition}

The coordinate calculation below shows that this object does indeed satisfy the requirements of a nonlinear splitting.  We will denote the natural induced coordinates on $TTM$  by the tuple $(x^i,y^{\alpha},v^i,w^{\alpha},X^i,Y^{\alpha},V^i,W^{\alpha})$. Then,
\begin{eqnarray*}
\nonumber &&\vilms{P}_v(x^i,y^{\alpha},v^i,w^{\alpha},X^i,Y^{\alpha},V^i,W^{\alpha}) \\
\nonumber &&=\sigma \circ TP_v (x^i,y^{\alpha},X^i,Y^{\alpha},v^i,w^{\alpha},V^i,W^{\alpha})\\
\nonumber &&=\sigma \Big(x^i,y^{\alpha},0, Y^{\alpha}-h^{\alpha}(x,y,X),v^i,w^{\alpha},0, \\&& \hspace*{4cm}  W^{\alpha}-\frac{\partial h^{\alpha}}{\partial x^j}(x,y,X)v^j-\frac{\partial h^{\alpha}}{\partial y^{\beta}}(x,y,X)w^{\beta}-\frac{\partial h^{\alpha}}{\partial v^j}(x,y,X)V^j\Big) \\
\nonumber &&= \Big(x^i,y^{\alpha},v^i,w^{\alpha},0, Y^{\alpha}-h^{\alpha}(x,y,X),0,\\ && \hspace*{4cm} W^{\alpha}-\frac{\partial h^{\alpha}}{\partial x^j}(x,y,X)v^j-\frac{\partial h^{\alpha}}{\partial y^{\beta}}(x,y,X)w^{\beta}-\frac{\partial h^{\alpha}}{\partial v^j}(x,y,X)V^j\Big).
\end{eqnarray*}

One readily verifies that the map $\vilms{h}:(T\pi)^*TN\rightarrow TTM$ is then given by 
\begin{eqnarray*}
\nonumber &&\vilms{h}(x^i,y^{\alpha},v^i,w^{\alpha},X^i,V^i)= \\  && \hspace*{-7mm} (x^i,y^{\alpha},v^i,w^{\alpha}, X^i ,h^{\alpha}(x,y,X), V^i ,\frac{\partial h^{\alpha}}{\partial x^j}(x,y,X)v^j+\frac{\partial h^{\alpha}}{\partial y^{\beta}}(x,y,X)w^{\beta}+\frac{\partial h^{\alpha}}{\partial v^j}(x,y,X)V^j).
\end{eqnarray*}

We can express this a bit more graphically by making use of vector fields an their lifts. The complete and vertical lifts of a horizontal lift by a nonlinear splitting are:
\begin{eqnarray*}
	\nonumber h\left(\frac{\partial}{\partial x^j}\right)&=&\frac{\partial}{\partial x^j}+h^{\alpha}(x,y,e_j)\frac{\partial}{\partial y^{\alpha}}, \\
	\clift{\left(h\left(\frac{\partial}{\partial x^j}\right)\right)}&=&\frac{\partial}{\partial x^j}+h^{\alpha}(x,y,e_j)\frac{\partial}{\partial y^{\alpha}}+\left(\frac{\partial h^{\alpha}}{\partial x^k}v^k+\frac{\partial h^{\alpha}}{\partial y^{\beta}}w^{\beta}\right)\bigg\rvert_{(x,y,e_j)}\frac{\partial}{\partial w^{\alpha}}, \\
	\vlift{\left(h\left(\frac{\partial}{\partial x^j}\right)\right)}&=&\frac{\partial}{\partial v^j}+h^{\alpha}(x,y,e_j)\frac{\partial}{\partial w^{\alpha}}	.
\end{eqnarray*}
The horizontal lifts of the  coordinate vector fields $\frac{\partial}{\partial x^j}$ and $\frac{\partial}{\partial v^j}$ on $TM$ by means of the Vilms nonlinear splitting are:
\begin{eqnarray*}
	\vilms{\frac{\partial}{\partial x^j}}&=&\frac{\partial}{\partial x^j}+h^{\alpha}(x,y,e_j)\frac{\partial}{\partial y^{\alpha}}+\left(\frac{\partial h^{\alpha}}{\partial x^k}v^k+\frac{\partial h^{\alpha}}{\partial y^{\beta}}w^{\beta}\right)\bigg\rvert_{(x,y,e_j)}\frac{\partial}{\partial w^{\alpha}}, \\
	\vilms{\frac{\partial}{\partial v^j}}&=&h^{\alpha}(x,y,0)\frac{\partial}{\partial y^{\alpha}}+\frac{\partial}{\partial v^j}+\left(\frac{\partial h^{\alpha}}{\partial x^k}v^k+\frac{\partial h^{\alpha}}{\partial y^{\beta}}w^{\beta}+\frac{\partial h^{\alpha}}{\partial v^j}\right)\bigg\rvert_{(x,y,0)}\frac{\partial}{\partial w^{\alpha}}.
\end{eqnarray*}
We conclude that for a nonlinear splitting $\left(h\left(\frac{\partial}{\partial x^j}\right)\clift{\right)}=\left(\clift{\frac{\partial}{\partial x^j}}\vilms{\right)}$ but $\left(h\left(\frac{\partial}{\partial x^j}\right)\vlift{\right)}\neq \left(\vlift{\frac{\partial}{\partial x^j}}\vilms{\right)}$.

{\bf An application.} Mechanical systems with rigid bodies are often subjected to nonholonomic constraints. These are nonintegrable constraints that depend on the velocities of the system. They appear, for example, in mechanical systems where wheels are supposed to roll without slipping, or when the system is prohibited from moving in certain directions (such as the motion of a skate). 

The literature on the case  where these nonholonomic constraints are  linear in the velocities is vast. However, many papers (see e.g.\ \cite{Cortes,deleonnonhol,Fasso,Olga,WFD} for a selection) also discuss nonlinear constraints. For simplicity,  we consider here  affine constraints, as in e.g. \cite{BKMM}. In that case, there exist configuration coordinates $(x^i,y^\alpha)$ for the nonholonomic  system, such that  the constraints can be written in the form
\[
{\dot y}^\alpha + A^\alpha_i(x,y) {\dot x}^i = A^\alpha_0(x,y).
\]
The geometric interpretation of such constraints in \cite{BKMM} is as follows: the presence of the coordinates $(x^i,y^\alpha)$ indicate that the configuration manifold of the mechanical system is the total space  of a fibre bundle $\pi: M\to N$. The functions  $A^\alpha_i$ then represent the local coefficients of an Ehresmann connection  on that bundle. Finally 
\[A_0 = A^\alpha_0(x,y)\fpd{}{y^\alpha}
\]
is a given $\pi$-vertical vector field on $M$. 

The point we would like to make is that, in the current set-up, we can interpret the constraints as the submanifold $\mathcal H$ of $TM$ that is the image of the nonlinear, but affine, splitting given by
\[
h (x^i,y^\alpha,v^i) = (x^i,y^\alpha,v^i,w^\alpha = - A^\alpha_i(x,y) v^i + A^\alpha_0(x,y) ).
\]
It is clear that, in general,  when a nonlinear splitting $h: \pi^*TN \to TM$ is an affine map, its linear part $(m,v_n)\mapsto h(m,v_n)-h(m,0_n)$, or 
\[
(x^i,y^\alpha,v^i) \mapsto  (x^i,y^\alpha,v^i,  - A^\alpha_i(x,y) v^i ), 
\]
is an Ehresmann connection on $\pi: M\to N$.
We will come back to the example of affine nonholonomic constraints, and the curvature of such an affine nonlinear splitting in Section~\ref{seccurvature}.

\section{Homogeneous nonlinear splittings and Ehresmann connections} \label{sec4}

In this section we discuss conditions for nonlinear splittings to become Ehresmann connections. We will also deal with a case `in between': that of a homogeneous nonlinear splitting.

\begin{definition}
A nonlinear splitting $h$ on $\pi:M \rightarrow N$ is   homogeneous   if it is positive homogeneous of degree 1, that is, if
	\[
	h(m,\lambda v_n)=	\lambda h(m,v_n), \qquad \forall \lambda \in \R^+.
	\]
\end{definition}

In terms of the local expression, the condition means that the coefficients $h^\alpha$ are $1^+$-homogeneous functions, $
	h^\alpha(x,y,\lambda v)=	\lambda h^\alpha(x,y,v)
$. In e.g.\ \cite{Szilasi} one may find Euler's theorem, which states that this is equivalent with the property
\[
v^i\fpd{h^\alpha}{v^i} = h^\alpha.
\]

 Because of the nonlinear nature of the splitting $h$, the procedure by which we define the horizontal lift of a curve  is not a geometric operation in the following sense:   the image of any of the horizontal lifts of a reparametrized curve might not be the same (viewed as a point set) as the image of the horizontal lift of the curve itself. Homogeneous splittings, however, do exhibit this geometric property.

\begin{proposition} \label{homprop} 
	Let $h$ be a homogeneous splitting on the fibre bundle $\pi:M \rightarrow N$. Then, for any horizontal lift of a curve $c_n(t)$ in $N$ and any positive reparametrization $\tilde{c}(s) = c_n(\theta(s))$ of $c_n(t)$ there exists a horizontal lift of $\tilde{c}(s)$, which has the same image  (as a point set) as the horizontal lift $c^h_m(t)$ of $c(t)$ to $m$.
\end{proposition}
\begin{proof} Let $m=(y_0^\alpha)$. The horizontal lift $c^h_m(t) = (x^i(t),y^\alpha(t))$ of ${c}_n(t)=(x^i(t))$ can by found by solving the initial value problem 
	\begin{eqnarray}
	\nonumber \dot{y}^{\alpha}(t)&=&h^{\alpha}(x(t),y(t),\dot{x}(t)), \\
	\nonumber {y}^{\alpha}(0)&=&y_0^{\alpha}.
	\end{eqnarray}
We will denote by $s_0$ the parameter value where $\theta(s_0)=0$. We consider the horizontal lift $\tilde{c}^h(s) = ({\tilde x}^i(s),{\tilde y}^\alpha(s))$ of ${\tilde c}(s) = ({\tilde x}^i(s)) = (x^i(\theta(s)))$ that corresponds to the initial value problem
\begin{eqnarray}
\nonumber {\tilde y}^{\alpha}{\,}'(s)&=&h^{\alpha}(\tilde{x}(s),\tilde{y}(s),\tilde{x}'(s)), \\
\nonumber \tilde{y}^{\alpha}(s_0)&=& y_0^{\alpha}.
\end{eqnarray}
Due to the  positive-homogeneity of $h^\alpha$ we see that
\[ \tilde{y}^{\alpha}{\,}'(s) =  h^{\alpha}\big(x(\theta (s)), y(s),{\dot x}(\theta(s))\theta '(s)\big)=\theta '(s)h^{\alpha}\big(x(\theta (s)), y(\theta(s)),{\dot x}(\theta(s))\big) = \theta'(s) {\dot y}^\alpha (\theta(s)), 
\]
which means that $\big(\tilde{y}^{\alpha}(s) - y^{\alpha}(\theta (s))\big)' = 0$. Since they both coincide at $s=s_0$, we get $\tilde{y}^{\alpha}(s) = y^{\alpha}(\theta (s))$, and the statement follows.
\end{proof}

In what follows, we  need the Liouville vector field on $TM$. It can be defined as the map $\Delta: w \mapsto (w,w\vlift)$, where $\vlift. : TM \times_M TM \to TTM$ stands for the vertical lift $(w_1,w_2\vlift)\in T_{w_1}TM$, with \[
(w_1,w_2\vlift)f = \frac{d}{dt}f(w_1+tw_2)\mid_{t=0}.
\] With this, the Liouville vector field becomes in natural coordinates $(q^a,u^a)$ on $TM$, $\Delta = u^a \partial/\partial u^a$. 

The vertical lift can be used to identify the set of vertical vector fields on $TM$ with the set $\vectorfields{\tau}$ of `vector fields along $\tau$'. These are sections of the pullback bundle $\tau^*TM \to TM$, and they can be regarded as maps $\zeta:TM \to TM$ with the property $\tau\circ\zeta = \tau$. We may therefore write them as $\zeta = \zeta^a(q,u)\partial/\partial q^a$. The corresponding vertical vector fields on $TM$,  $\vlift\zeta:w \mapsto (w,\zeta(w)\vlift)$ is then $\vlift\zeta = \zeta^a(q,u)\partial/\partial u^a$. In the special case of the so-called canonical vector field $\T$ along $\tau$ (the map $w\mapsto w$) its vertical lift is the Liouville vector field $\Delta$.

Consider a (general) nonlinear splitting. Similar to the definition of the Liouville vector field, we may introduce two vector fields on $TM$: 
\[
{\Delta}_h: TM \to TTM,\, w\mapsto (w,P_h(w)\vlift) \qquad \mbox{and} \qquad {\Delta}_v: TM \to TTM,\, w\mapsto (w,P_v(w)\vlift).
\]
Then $\Delta=\Delta_h+ \Delta_v$ and
\[
\Delta_h = v^i\fpd{}{v^i} + h^\alpha(x,y,v) \fpd{}{w^\alpha} \qquad \mbox{and} \qquad {\Delta}_v = (w^\alpha - h^\alpha(x,y,v))\fpd{}{w^\alpha}.
\]
An other vector field of interest is $\Delta_0(v) = (v,h(0)\vlift)$, with
\[
\Delta_0 = h^\alpha(x,y,0) \fpd{}{w^\alpha}.
\]

\begin{proposition} \label{Ehresmann1} A  nonlinear splitting $h: \pi^*TN \to TM$ is homogeneous if and only if one of the following equivalent characterizations are satisfied:
\begin{itemize} \item[(1)] $[\Delta,\Delta_h] = 0$.
 \item[(2)] The Liouville vector field $\Delta \in \vectorfields{TM}$ on $M$ is tangent to ${\mathcal H}$.
\end{itemize}
\end{proposition}
\begin{proof} (1)
One easily verifies that 
\[
[\Delta,\Delta_h] = \left( v^i \fpd{h^\alpha}{v^i} -   h^\alpha \right)\fpd{}{w^\alpha}.
\] As we explained before, we may conclude from
the condition $v^i \fpd{h^\alpha}{v^i} = h^\alpha$ that $h^\alpha$ is a 1-homogeneous function in $v^i$. 

(2) When $\Delta$ is tangent to $\mathcal H$, then $0=\Delta(w^\alpha -h^\alpha) = w^\alpha - v^i \fpd{h^\alpha}{v^i}$, whenever $w^\alpha=h^\alpha$. We obtain again that $v^i \fpd{h^\alpha}{v^i} = h^\alpha$.
\end{proof}

Since $\Delta_v=\Delta- \Delta_h$, we could also have written $[\Delta ,\Delta_v] = 0$ or $[\Delta_h,\Delta_v]=0$ in the first item of the Proposition.

We end this section with a few characterizations of when a nonlinear splitting is an Ehresmann connection, in case the map $h$ is smooth on the whole of $\pi^*TN$. 
\begin{proposition} \label{Ehresmann2}
 A smooth nonlinear splitting $h: \pi^*TN \to TM$ is an Ehresmann connection on $\pi$ if and only if one of the following equivalent characterizations are satisfied:
\begin{itemize} \item[(1)] $[\Delta ,\Delta_h] = 0$.
 \item[(2)] The Liouville vector field $\Delta \in \vectorfields{TM}$ on $M$ is tangent to ${\mathcal H}$.
\item[(3)] The Liouville vector field $\bar\Delta \in \vectorfields{TN}$ on $N$ satisfies $\vilms{\bar\Delta}=\Delta_h$.
\end{itemize}
\end{proposition}
\begin{proof} The statements (1) and (2) are a corollary of the previous proposition:  If a function is of class ${\mathcal C}^1$ and positive-homogeneous of degree 1 in $v^i$, then it is a
linear function  in $v^i$ (see \cite{Szilasi}).

(3) Since $\bar\Delta = v^i \partial/\partial v^i$, the expression for $\vilms{\bar\Delta}$ is 
\[
\vilms{\bar\Delta} =  h^\alpha(x,y,0) \fpd{}{y^\alpha}+v^i\fpd{}{v^i} +  \left(\frac{\partial h^{\alpha}}{\partial x^k}v^k+\frac{\partial h^{\alpha}}{\partial y^{\beta}}w^{\beta}+ \frac{\partial h^{\alpha}}{\partial v^j}v^j\right)\bigg\rvert_{(x,y,0)}\frac{\partial}{\partial w^{\alpha}}.
\]
When compared to $\Delta_h$ we get that $ h^\alpha(x,y,0)=0$. From this,   also 
\[
\frac{\partial h^{\alpha}}{\partial x^k}(x,y,0) = \frac{\partial h^{\alpha}}{\partial y^{\beta}}(x,y,0) = 0.
\]
With that, the remaining condition becomes 
\[
\frac{\partial h^{\alpha}}{\partial v^j}(x,y,0)v^j =h^\alpha(x,y,v).
\]
This shows that $h^\alpha$ is a linear function in $v^i$.
\end{proof}

We will come back to homogeneous nonlinear splittings and their appearance in Finsler geometry in Section~\ref{secoutlook}.

\section{Subduced Lagrangians} \label{subduced}

In what follows we will often consider the tangent bundle $\tau: TM \to M$ of a differentiable manifold $M$. Coordinates $(q^a)$ on $M$ induce coordinates $(q^a,u^a)$ on $TM$. We refer to e.g.\ \cite{CrampinPirani,deleonrodrigues} for the definitions and elementary properties of the next few concepts.

A {\em second-order ordinary differential equations field} $\Gamma$ on $M$ (from now on  `a \sode\ on $M$', in short) is a vector field on $TM$ with the property that all its integral curves $\gamma(t)$ in $TM$ are (tangent) lifted curves $\dot c(t)$ of curves $c(t)$ on $M$ (the so-called base integral curves of $\Gamma$). A \sode\ is  locally given by
\[
\Gamma={u}^a\fpd{}{q^a}+f^a\fpd{}{u^a},
\]
from which it is clear that its base integral curves $c(t) = (q^a(t))$ satisfy
\begin{equation} \label{hulpke}
\ddot{q}^a=f^a(q,\dot{q}).
\end{equation}

Our main area of application is in Lagrangian mechanics. The equations of motion are then given by
\[
\frac{d}{dt}\left(\fpd{L}{u^a} \right) -\fpd{L}{q^a}=0.
\]
A Lagrangian function $L\in\cinfty{TM}$ is {\em regular} when its Hessian with respect to fibre coordinates,
\[
\spd{L}{u^a}{u^b}
\]
considered as a symmetric matrix, is everywhere non-singular. In that case, the Euler-Lagrange equations can be rewritten in the form
(\ref{hulpke}), which indicates that the solutions are the base integral curves of a \sode\ $\Gamma_L$ on $M$. In e.g.\ \cite{Inv} it is shown that this so-called {\em Euler-Lagrange vector field} $\Gamma_L$ is a vector field on $TM$ that is completely determined by the fact that it is a \sode\ and  that it satisfies 
\begin{equation}\label{ELsode}
\Gamma_L (\vlift{X}(L)) - \clift{X}(L) = 0,\qquad \forall X \in\vectorfields{M}.
\end{equation}
Here $\clift{X} = X^a \partial/\partial q^a + (\partial X^b/ \partial q^a) \partial/\partial u^a$ and $\vlift{X} = X^a\partial/\partial u^a$ stand, respectively, for the complete lift and vertical lift of a vector field $X = X^a \partial/\partial q^a$ on $M$.

The goal of this section is to study the basic properties of an important subclass of nonlinear splittings.
\begin{definition} Let $L$ be a regular Lagrangian on the total space $M$ of the fibre bundle $\pi: M\rightarrow N$. The Lagrangian is fibre-regular when its Hessian  with respect to fibre coordinates is non-degenerate, that is, when
\[
\det\frac{\partial^2 L}{\partial w^{\alpha}\partial w^{\beta}}\neq 0.
\]
\end{definition}
We may also give an intrinsic definition, in terms of the Legendre transformation ${\rm Leg}: TM \to T^*M, (q^a,u^a) \to (q^a, p_a = \partial L/\partial u^a)$. If we consider the Legendre transformation from the vertical bundle
of $\pi$ on its dual bundle then, one can see that $L$ is fibre-regular if this
transformation is a local diffeomorphism.

When the Lagrangian $L$ is fibre-regular, the Implicit Function Theorem 
guarantees the local existence of a map $h:\pi^*TN\rightarrow TM$, as the solution of
\begin{equation}\label{definingrelation}
	\frac{\partial L}{\partial w^{\alpha}}\circ h=0.
\end{equation}
In the context of Lagrangian systems  we will mainly work with local nonlinear splittings.

\begin{definition} \label{induced} The  nonlinear splitting induced by the fibre-regular Lagrangian $L$ is the map $h:\pi^*TN\rightarrow TM$  determined by the relation (\ref{definingrelation}).
\end{definition}

Any vector field $Y$ on $M$ can be $\tau$-vertically lifted to a vector field $\vlift{Y}$ on $TM$. In case $Y$ is $\pi$-vertical (i.e.\ in case it satisfies $T\pi\circ Y = 0$) its vertical lift $\vlift{Y}$ is a combination of the basis vector fields $\partial/\partial w^\alpha$.  The defining relation of the splitting can then also be written as 
\[
\vlift{Y}(L) \circ h =0, \qquad \forall \mbox{$\pi$-vertical $Y$}.
\]

The following natural question arises:  When are horizontal curves solutions of the Euler-Lagrange equations of $L$? Or, formulated differently: When do solutions, with horizontal initial velocity remain horizontal? 

We recall first an observation about a submanifold $S$ of a manifold $Q$, with inclusion  $\iota: S \to Q$. The following statements can be found in e.g.\ \cite{Lee}. A vector field $X$ on $Q$ is tangent to a submanifold $S$ if and only if $X(f)$ vanishes on $S$, for every function $f$ on $Q$ that vanishes on $S$. If a vector field $X$ is tangent to $S$, then there exist a unique vector field $Y$ on $S$ that is $\iota$-related to $X$. If $S$ is closed the integral curves of $X$ that start in $S$ remain  in $S$. When  $S$ is not closed, the result only holds locally. Indeed, if $x(t)$ is the integral curve of $X$ through $\iota(y_0)$, with $y_0\in S$, then the Picard-Lindel\"of theorem ensures that it must be of the form $(\iota\circ y)(t)$, where $y(t)$ is the integral curve of $Y$ through $y_0$. In the next Proposition, we will apply these statements to the case where $Q=TM$, $S={\mathcal H}$ and $X=\Gamma_L$ is the Euler-Lagrange \sode\, of $L$. For the functions $f$ we take the functions $\partial L/\partial w^\alpha$  that determine the submanifold ${\mathcal H}$. A version of the  next result can also be found in \cite{popescu1}, but in a somewhat different context.

\begin{proposition}\label{symmetryprop}
	Let $h$ be the nonlinear splitting of a fibre-regular Lagrangian $L$. Then, the Euler-Lagrange field $\Gamma_L$ of $L$ is tangent to ${\mathcal H}$ if and only if
	\begin{equation}\label{symmetrycondition}
		\clift{Y}(L)\circ h=0,	
	\end{equation}
for any vector field $Y$ on $M$ that is $\pi$-vertical. Under the assumption that (\ref{symmetrycondition}) is satisfied, the function $\bar L:=L\circ h$ determines a Lagrangian  on $N$. If it is regular, the base integral curves of its Euler-Lagrange field ${\bar\Gamma}_{\bar L}$ on $N$ are   the projections of horizontal base integral curves of $\Gamma_L$.
\end{proposition}

\begin{proof} In view of equation (\ref{ELsode}), the condition (\ref{symmetrycondition}) is  equivalent with $\Gamma (\vlift{Y}(L)) \circ h=0$, for all $\pi$-vertical $Y$. This  shows that $\Gamma$ is tangent to ${\mathcal H}={\rm Im}\,h$. 

We may therefore conclude that the integral curve $\gamma(t)$ of $\Gamma$ (as a curve in $TM$) that starts at an element $h(m_0,v_0)$ of ${\mathcal H}$ remains in ${\mathcal H}$. Since $\Gamma$ is a \sode\ we know that each integral curve is a lifted curve. For this reason, the integral curves that start at a horizontal element are (tangent) lifts of horizontal lifts, and $\gamma(t)={\dot c}^h_{m_0}(t)$. 

Under the condition (\ref{symmetrycondition}) the function $\bar L =L \circ h$ on $\pi^*TN$ (in principle depending on coordinates $(x^i,y^\alpha,v^i)$) can in fact be thought of as a function on $TN$. Indeed,  we get that
\[
\fpd{\bar L}{y^\alpha}=\fpd{L}{y^\alpha}\circ h + \left(\fpd{L}{w^\beta}\circ h\right) \fpd{h^\beta}{y^\alpha}. 
\]
The first term vanishes because of the condition (\ref{symmetrycondition}) and the second because of the definition of the submanifold (in both cases using $Y=\partial/\partial y^\alpha$). 
This shows, that the composition $L\circ h$ does not depend on the coordinates $y^{\alpha}$ and that it therefore  restricts to a function on $TN$.

For later reference we list all first- and second-order derivatives of $\bar L$:
\begin{eqnarray*}
\frac{\partial \bar L}{\partial x^i}&=&\frac{\partial L}{\partial x^i}\circ h+\left(\frac{\partial L}{\partial w^{\alpha}}\circ h\right)\frac{\partial h^{\alpha}}{\partial x^i}=\frac{\partial L}{\partial x^i}\circ h, \label{firstsetofrelation0}\\
\frac{\partial \bar L}{\partial v^i}&=&\frac{\partial L}{\partial v^i}\circ h+\left(\frac{\partial L}{\partial w^{\alpha}}\circ h\right)\frac{\partial h^{\alpha}}{\partial v^i}=\frac{\partial L}{\partial v^i}\circ h, \label{firstsetofrelation1}\\
\frac{\partial^2 \bar L}{\partial v^i\partial x^j}&=&\frac{\partial ^2 L}{\partial v^i\partial x^j}\circ h+\left(\frac{\partial^2L}{\partial v^i\partial w^{\alpha}}\circ h\right)\frac{\partial h^{\alpha}}{\partial x^j}, 
\\
\frac{\partial^2\bar L}{\partial v^i\partial v^j}&=&\frac{\partial ^2L}{\partial v^i\partial v^j}\circ h+\left(\frac{\partial^2L}{\partial v^i\partial w^{\alpha}}\circ h\right)\frac{\partial h^{\alpha}}{\partial v^j}. 
\end{eqnarray*}
Finally, we remark that 
\[
0= 
\frac{\partial^2\bar L}{\partial v^i\partial y^{\alpha}}=\frac{\partial ^2L}{\partial v^i\partial y^{\alpha}}\circ h+\left(\frac{\partial^2L}{\partial v^i\partial w^{\beta}}\circ h\right)\frac{\partial h^{\beta}}{\partial y^{\alpha}} .
\]
Consider now a base integral curve of $\Gamma$ of the type $c^h_{m_0}(t)=(x^i(t),y^\alpha(t))$. Then ${\dot y}^\alpha = h^\alpha$ and 
\[
{\ddot y}^\alpha = \frac{\partial h^{\alpha}}{\partial x^j} {\dot x}^j    +\frac{\partial h^{\alpha}}{\partial v^j} {\ddot x}^j  +\frac{\partial h^{\alpha}}{\partial y^{\beta}} {\dot y}^\beta.
\]
 Let $n_0=\pi(m_0)$ and $c_{n_0}(t)=\pi(c_{m_0}(t))=(x^i(t))$. We verify that $c_{n_0}$ is a base integral curve of the Euler-Lagrange \sode\ $\bar\Gamma$ of $\bar L$: 
\begin{eqnarray*}
\frac{d}{dt}\left( \fpd{\bar L}{v^i}\right) - \fpd{\bar L}{x^i}&=&\frac{\partial^2 \bar L}{\partial v^i\partial x^j}{\dot x}^j+\frac{\partial^2 \bar L}{\partial v^i\partial v^j}{\ddot x}^j -\frac{\partial \bar L}{\partial x^{i}} \\&=& \left( \frac{\partial^2 L}{\partial v^i\partial x^j}\circ h+\left(\frac{\partial^2\partial L}{\partial v^i\partial w^\alpha}\circ h\right)\frac{\partial h^{\alpha}}{\partial x^j} \right){\dot x}^j \\&&  \hspace*{3.2cm} +      \left( \frac{\partial^2L}{\partial v^i\partial v^j}\circ h+\left(\frac{\partial^2   L}{\partial v^i\partial w^\alpha }\circ h\right)\frac{\partial h^{\alpha}}{\partial v^j} \right) {\ddot x}^j    - \frac{\partial L_1}{\partial x^i}\circ h 
\\&=&  \left( \frac{\partial^2 L}{\partial v^i\partial x^j}  {\dot x}^j +         \frac{\partial^2L}{\partial v^i\partial v^j} {\ddot x}^j +  \frac{\partial^2L}{\partial v^i\partial y^{\alpha}}  {\dot y}^\alpha - \frac{\partial L}{\partial x^i} \right) \circ h 
\\&&\hspace*{3.2cm}  +\left(\frac{\partial^2 L}{\partial v^i\partial w^\alpha}\circ h\right) \left(\frac{\partial h^{\alpha}}{\partial x^j} {\dot x}^j    +\frac{\partial h^{\alpha}}{\partial v^j} {\ddot x}^j  +\frac{\partial h^{\alpha}}{\partial y^{\beta}} {\dot y}^\beta \right)
\\&=&  \left( \frac{\partial ^2 L}{\partial v^i\partial x^j}  {\dot x}^j +         \frac{\partial^2L}{\partial v^i\partial v^j} {\ddot x}^j +  \frac{\partial^2L}{\partial v^i\partial y^{\alpha}}  {\dot y}^\alpha - \frac{\partial L}{\partial x^i} \right) \circ h 
+\left(\frac{\partial^2 L}{\partial w^\alpha \partial v^i}\circ h\right) {\ddot y}^\alpha.
\\&=&\left(\frac{d}{dt}\left( \fpd{L}{v^i}\right) - \fpd{L}{x^i}\right)\circ h = 0. 
\end{eqnarray*}
\end{proof}

\begin{definition}
	Assume that a nonlinear splitting $h$ on $\pi:M\rightarrow N$ is induced by the fibre-regular Lagrangian $L$. If the condition (\ref{symmetrycondition}) of Proposition \ref{symmetryprop} is satisfied, we call the Lagrangian $\bar L$ on $N$ given by $L\circ h$ the  subduced Lagrangian of $L$ through $\pi$. 
\end{definition} 
In what follows we will always assume that the subduced Lagrangian is regular.

{\bf Remark.} 	A smooth map between manifolds  $\pi:M\rightarrow N$  is called a submersion if its tangent map is surjective at any point. 
The following property can be found in \cite{kossowski}:
	{\em Assume that $\pi:M\rightarrow N$ is a surjective submersion and $\Gamma$ is a \sode\, on $M$. If $\Gamma$ is $T\pi$-related to some vector field $\bar\Gamma$ on $N$ (that is, if 	$T(T\pi) \circ \Gamma = \bar\Gamma \circ T\pi$), then $\bar\Gamma$ is also a \sode\, (on $N$).} 
For this reason, a \sode\
 $\Gamma$ is called {\em submersive} if there exists a surjective submersion $\pi:M\rightarrow N$ such that $\Gamma$ is $T\pi$-related to a vector field $\bar\Gamma$ on $N$. Submersive \sode s have been extensively investigated in the literature, see for instance \cite{kossowski,willysubmersive}. If such a surjective submersion exists, $\pi$ can locally be expressed as
$\pi(x^i,y^{\alpha})=(x^i)$
and the base integral curves of $\Gamma$ satisfy
\begin{eqnarray}
	\nonumber \ddot{x}^i&=&f^i(x,\dot{x}) \\
	\nonumber \ddot{y}^{\alpha}&=&f^{\alpha}(x,y,\dot{x},\dot{y}).
\end{eqnarray}
The first set of equations constitutes a decoupled subsystem of $\Gamma$ with fewer variables. This subsystem represents the \sode\ $\bar\Gamma$ and, therefore, each base integral curve of $\Gamma$ projects through $\pi$ to a base integral curve of $\bar\Gamma$. 
 Although Proposition~\ref{symmetryprop} seemingly says something similar, it is, in the current  generality, not true that the Euler-Lagrange field of $L$ is submersive. The reason is that  Proposition~\ref{symmetryprop}   only makes a statement about the projection of horizontal integral curves, and not about the projection of base integral curves in general. In \cite{popescu1,popescu2}, $\pi: M\to N$ is called a Lagrangian submersion if a subduced Lagrangian exists, even though the two Lagrangian \sode s do not submerse in the sense of \cite{kossowski}. In the next sections, we will present   two situations where we can relate properties of nonlinear splittings to submersive \sode s. 

\section{Principal splittings and symmetry reduction of Lagrangian systems} \label{secprincipal}

A nonlinear splitting, induced by a Lagrangian on a fibre bundle does not necessarily give rise to a subduced Lagrangian on the base manifold of the bundle. We will see now that, in the presence of symmetries, the existence of such a subduced function can be guaranteed. The next definition can be thought of as a generalization of a principal connection on a principal bundle.

Let $G$ be a connected Lie group  with Lie algebra $\mathfrak{g}$. Assume that the manifold $M$ comes equipped with a free and proper Lie group action $\Phi:G\times M\rightarrow M$. Then $M$ is the total space of  a principal $G$-bundle $\pi:M\rightarrow N=M/G$. 

The pullback bundle  $\pi^*T(M/G)$ also carries a natural $G$-action $\bar\Phi: G \times \pi^*T(M/G) \to \pi^*T(M/G)$, with  ${\bar\Phi}_g(m,v_n) = (\Phi_g(m),v_n)$. Moreover, the maps $T\Phi_g$ induce a $G$-action on $TM$. We will denote the corresponding principal fibre bundle with $\pi^{TM}: TM \to (TM)/G$.   

The  vertical space of $\pi: M \to M/G$, $V\pi={\rm Ker T\pi} \subset TM$, can be identified with $M\times\mathfrak{g}$ if we make  use of the trivialization $(m,\xi)\mapsto \tilde{\xi}_{\alpha}(m)$. Herein is $\tilde\xi \in  \vectorfields{M}$ the infinitesimal generator of a $\xi \in \mathfrak{g}$.  The action $T\Phi_g$ of $TM$ is then in agreement with the action $g\cdot(m,\xi)= (\Phi_gm,{\rm Ad}_g\xi)$ on $M\times \mathfrak{g}$. For this reason, all manifolds in the short exact sequence
	\[
	\begin{tikzcd}
	0\arrow{r} & M\times\mathfrak{g}\arrow{r}{} & TM\arrow{r}{\mu} & \pi^*T(M/G)\arrow{r} & 0,
	\end{tikzcd}
	\]
have a $G$-action and the sequence itself is $G$-equivariant. We will be interested in {\em equivariant} nonlinear splittings of this sequence.

\begin{definition} A   nonlinear splitting $h:\pi^*T(M/G)\rightarrow TM$ on a principal bundle $\pi:M\rightarrow N=M/G$
 is principal if it is  equivariant: 
\[h({\bar\Phi}_g(m,v_n))=T\Phi_g(h(m,v_n)).
	\]
The (nonlinear) map $\omega:TM\rightarrow\mathfrak{g}$, defined by the requirement 
\[
P_v(w_m)=\widetilde{\omega(w_m)}(m)
\] is the left nonlinear splitting of a principal nonlinear splitting $h$.
\end{definition}
It is easy to see that the equivariance property of $h$ translates to the properties
\[
P_v \circ T\Phi_g  = T\Phi_g \circ P_v \qquad \mbox{and}\qquad \omega \circ T\Phi_g = {\rm Ad}_g \circ \omega 
\]
for $P_v$ and $\omega$. 

If, as before, $\pi:(x^i,y^\alpha) \to (x^i)$ are local adapted coordinates and $\{E_\alpha\}$ is a basis of $\mathfrak{g}$ we can decompose the fundamental vector fields ${\tilde E}_\alpha$ as
\[
{\tilde E}_\alpha = K^\alpha_\beta(x,y) \fpd{}{y^\beta}.
\]
These vector fields form a frame for the set of $\pi$-vertical vector fields on $M$. A local coordinate expression for  $\omega$ is then
\[
\omega(x^i,y^{\alpha},v^i,w^\alpha)=\omega^{\beta}E_{\beta}=(K^{-1})_{\gamma}^{\beta}(w^{\gamma}-h^{\gamma})E_{\beta}.
\]

Since a principal connection is most frequently represented by its left splitting, we  give a characterization of when a nonlinear splitting is a connection, in terms of $\omega$.

\begin{proposition} \label{principal}
A smooth principal nonlinear splitting  is a principal connection on $\pi$ if and only if $i_{\Delta_h} {d}\omega =0$. 
\end{proposition}
\begin{proof} In the above statement we consider $\omega = \omega^\alpha E_\alpha$ as a ${\mathfrak g}$-valued function on $TM$. Its exterior derivative is then the $\mathfrak g$-valued one-form on $TM$ with the property that, for $\zeta=\zeta^a \partial/\partial q^a + Z^b \partial /\partial u^b \in \vectorfields{TM}$, 
\[
{d}\omega (\zeta) = \left( \zeta^a\fpd{\omega^\alpha}{q^a} + Z^b\fpd{\omega^\alpha}{u^b}   \right) E_\alpha.
\]
For  $\Delta_h$ we have $\zeta^a = 0$, $Z^i=v^i$ and $Z^\alpha=h^\alpha$. Together with the above expression for $\omega^\alpha$ we obtain
\[
 {d}\omega (\Delta_h) = (K^{-1})^\alpha_\gamma\left( -v^i\fpd{h^\gamma}{v^i} + h^\gamma   \right) E_\alpha.
\]
Therefore, the property $ {d}\omega (\Delta_h)=0$ returns the condition $v^i \fpd{h^\alpha}{v^i} = h^\alpha$ that we have already discussed in the proof of Proposition~\ref{Ehresmann2}.
\end{proof}

 We have assumed that the Lie group $G$ is connected. It is well-known that in that case a function $f$ on $M$ is invariant under the action $\Phi$ if and only if ${\tilde\xi}(f) = 0$, $\forall \xi\in\la$. Likewise, a function $F$ on $TM$ is invariant under the induced $G$-action $T\Phi_g$ on $TM$ if and only if $\clift{\tilde \xi}(F)=0$. There exist analogous characterizations for tensor fields and vector fields. The next proposition gives an infinitesimal characterization of a principal nonlinear splitting. 

\begin{proposition}
A nonlinear splitting $h$ on a principal bundle is principal if and only if  $\Delta_h$ is an invariant vector field on $TM$, i.e.\ $[\clift{\tilde \xi},\Delta_h]=0$, $\forall \xi\in\la$.
\end{proposition}

 \begin{proof}
We show first that  $\Delta_h$ is invariant for a principal splitting:
\begin{eqnarray*}
\Delta_h (T\Phi_g (w_m)) &=& ( T\Phi_g (w_m),P_h(T\Phi_g (w_m))  \vlift) = ( T\Phi_g (w_m),T\Phi_g (P_h(w_m))  \vlift) \\&=& TT\Phi_g \big(( w_m, P_h(w_m)\vlift)\big) =   TT\Phi_g(\Delta_h (w_m)).
\end{eqnarray*}
The before last equality follows essentially because $T\Phi_g$ is  a linear bundle map $H: TM \to TM, (q^a,u^a)\mapsto (q^b, H^b_a(q)u^a)$: If we set $w_m=(q^a,u^a)$ and $P_h(w_m)=X_m=(q^a,X^a)$, then 
\[
(H(w_m),H(X_m)\vlift) = H^b_aX^a\fpd{}{u^b}\bigg\rvert_{(q^b,H^b_a u^a)} =TH\left( X^a\fpd{}{u^a} \bigg\rvert_{(q^a,u^a)}  \right) = TH\big((w_m, X_m\vlift)\big). 
\]
 If we run the same steps in the opposite direction we obtain from the invariance of $\Delta_h$ that $P_h(T\Phi_g (w_m)) = T\Phi_g (P_h(w_m)) $.
\end{proof}
Since $[\clift X,\Delta] = 0$ for any vector field $X$ on $M$, the property of the above proposition could also be written as $[\clift{\tilde \xi},\Delta_v]=0$. For later reference we express it in coordinates: If $\tilde\xi = \xi^\gamma K_\gamma^\beta \partial /\partial y^\beta$, then the condition becomes 
\begin{equation} \label{help}
   v^i \fpd{K^\alpha_\gamma}{x^i}(x,y)  + h^\beta(x,y,v)\fpd{K^\alpha_\gamma}{y^\beta}(x,y) -K^\beta_\gamma(x,y)\fpd{h^\alpha}{y^\beta}(x,y,v) =0, \qquad \forall (x,y,v).
\end{equation}

Our main example comes again from Lagrangian mechanics. Assume that a Lagrangian $L$ is invariant under the induced action of the Lie group $G$ on $TM$. If the Lie group is connected, this can be expressed as
\begin{equation} \label{inv}
 \clift{\tilde{\xi}}(L)=0, \qquad \forall \xi \in \mathfrak{g}.
\end{equation}

We recall from \cite{MR,MRS} that the map $J_L:TM\rightarrow \mathfrak{g}^*$ defined as
\[
\langle J_L(w_m),\xi \rangle =\vlift{\tilde{\xi}}(L)(w_m)
\]
is called the momentum map of $L$. For each $w_m\in TM$ we may define the restriction $J_L\mid_{w_m} \la \to \la^*, \xi\mapsto J_L(w_m+{\tilde \xi}(m))$. In \cite{Bavo,Bavo2} the Lagrangian is said to be $G$-regular if $J_L\mid_{w_m}$ is a diffeomorphism for each $w_m\in TM$. Since then
\[
\det \left(  \vlift{\tilde E}_\alpha\vlift{\tilde E}_\beta (L) \right) \neq 0,
\]
this notion coincides, in this case, with that of a fibre-regular Lagrangian from Section~\ref{subduced}. The corresponding nonlinear splitting induced by $L$ is then globally defined.

It is well known that $J$ is equivariant with respect to the action of $G$ on $TM$ and the coadjoint action of $G$ on $\mathfrak{g}^*$ given by
\[
\langle \xi, \ad_g^*\mu\rangle = \langle \ad_g\xi, \mu\rangle.
\]
For a fixed $\mu\in \mathfrak{g}^*$, the level set $N_\mu$ of momentum $\mu$ is invariant under the isotropy subgroup $G_\mu=\{g\in G~|~\ad_g^*\mu=\mu\}$ (see also \cite{MRS,Bavo,tomsrouth}). 

In the current situation of interest, the nonlinear splitting that is induced by $L$ is determined  by the relation $\vlift{\tilde \xi}(L)\circ h=0$. The corresponding submanifold ${\mathcal H} = {\rm Im}\,h$   can therefore be regarded as the level set of zero momentum, $\mu=0$. Since then $G_\mu=G$,  ${\mathcal H}$ is invariant under the whole group action, and as a consequence, the corresponding $h$ is a principal nonlinear splitting. 

\begin{proposition}
	Let $\pi: M\rightarrow M/G$ be a principal bundle and $L$ be a $G$-regular invariant Lagrangian on $M$ and $h$ its principal splitting. Then, the Euler-Lagrange field of $L$ is tangent to the image of $h$ and horizontal solutions (with zero momentum) can be projected to solutions of the Euler-Lagrange equations of $\bar L=L\circ h$.
\end{proposition}

\begin{proof} This follows immediately from Proposition~\ref{symmetryprop}, since the presence of symmetries (\ref{inv}) ensures that  condition (\ref{symmetrycondition}) is satisfied on the whole of $TM$. \end{proof}

The subduced Lagrangian $\bar L$ (a function on $T(M/G)$) has an interesting interpretation in the current situation. When the Lagrangian $L$ is $G$-invariant it can be identified with the `reduced Lagrangian' $l$. This is a function on $(TM)/G$ that can be implicitly determined from $L=l\circ \pi^{TM}$, where $\pi^{TM}: TM \to (TM)/G$ is the principal bundle one may associate with the $G$-action on $TM$. 

 In several papers (e.g.\ \cite{CMR,Inv,DMM}) it has been shown that the Euler-Lagrange equations of $L$ can be reduced to the so-called Lagrange-Poincar\'e equations of $l$. These equations are essentially associated to a vector field on $(TM)/G$. The point we would like to make is that $(TM)/G$ is (only) a Lie algebroid (the so-called Atiyah algebroid, see e.g.\ \cite{DMM}), and $(TM)/G$ can not be identified with the tangent manifold $T(M/G)$. Since $l$ is a function on $(TM)/G$, there is no obvious relation between the Euler-Lagrange field of the subduced Lagrangian $\bar L$ (a vector field on $T(M/G)$) and the Lagrange-Poincar\'e field of $l$ (a vector field on $(TM)/G$).

Besides Lagrange-Poincar\'e reduction, there exist, however, a second symmetry reduction method: Routh reduction. 
From the Euler-Lagrange equations of $L$ we know that
\[
\Gamma_L (\vlift{\tilde{E}}_{\alpha}(L))=0.
\]
This shows that the Euler-Lagrange vector field $\Gamma$ of $L$ is tangent to any level set $\vlift{\tilde{E}}_{\alpha}(L)=\mu_{\alpha}$, where  $\mu_{\alpha}$ can now be any arbitrary constants. Routh reduction takes optimal advantage of this observation. For the details we refer the reader to e.g.\ \cite{tomsrouth, Bavo, MRS}. Here it is enough to know that the reduced equations are completely determined by the so-called Routhian function on $N_\mu$ (the level set of momentum, corresponding to $\mu\in {\mathfrak{g}}^*$),
\[
{\mathcal R}^\mu = (L - \mu_\alpha \v^\alpha)\mid_{N_\mu}.
\]
This is a $G_\mu$-invariant function on $N_\mu$ and, for this reason, it can be identified with a function on $N_\mu/G_\mu$.
In this generality, it has been shown that solutions of the Euler-Lagrange equations of $L$ that remain on a specific level set $N_\mu$, can be seen to be solutions of `Routh  equations'  on $N_\mu/G_\mu$. In \cite{tomsrouth} these Routh equations have been computed to be of the form
\[
\frac{d}{dt}\left(\fpd{{\mathcal R}^\mu}{v^i}\right)-\fpd{{\mathcal R}^\mu}{x^i}
=-\mu_\alpha R^\alpha_{ij}v^j - \Lambda^A_i \fpd{{\mathcal R}^\mu}{\theta^A}.
\]
(Here $(x^i,v^i,\theta^A)$ are coordinates on $N_\mu/G_\mu$, where $\theta^A$ stand for coordinates on $G/G_\mu$. The precise meaning of the terms in the righthand side is not of importance to us here.)

In the case of current interest, where $\mu=0$, we immediately see that ${\mathcal R}^0=\bar L$ is our subduced Lagrangian, and that it can be thought of as a function on $N_0/G_0=T(M/G)$. From $G_\mu=G$ it also follows that no $\theta^A$-coordinates appear in the above Routh equations. They therefore simplify, indeed, to the Euler-Lagrange equations of $\bar L$: We conclude that at zero momentum the Routh equations show a variational nature.

We continue our investigation of principal nonlinear splittings. Recall that in Section~\ref{subduced} we have seen that the \sode s $\Gamma_L$ of a Lagrangian $L$ on $M$  and ${\bar\Gamma}_{\bar L}$ of  its subduced Lagrangian $\bar L$ on $N$ are not submersive in the sense of \cite{kossowski}. We will show now that, in the case of a principal splitting on $\pi: M\rightarrow M/G$, it is possible to start with a \sode\ $\bar\Gamma$ on $N=M/G$, and to  construct a \sode\ $\Gamma$ on $M$ that is submersive. In case of Ehresmann connections, this procedure is called  `unreduction' in \cite{unreduction}.  

We have already mentioned that $(TM)/G$ should not be confused with the tangent manifold $T(M/G)$. In fact, $T(M/G)$ is the quotient of $TM$ by the action of $TG$ (which is also a  Lie group): The principal bundle that corresponds to the action $\Psi=T\Phi: TG \times TM \to TM$ has projection $T\pi: TM \to T(M/G)$. We show now that the Vilms lift of a principal nonlinear splitting on $\pi: M \to M/G$ is a principal splitting on the $TG$-principal bundle $T\pi: TM \to T(M/G)$.

\begin{proposition}
The Vilms nonlinear splitting is a principal splitting on $T\pi:TM \to T(M/G)$.
\end{proposition}
\begin{proof} Let $v_g\in TG$ and $\xi=g^{-1}v_g \in \la$. The statement will follow if we can prove that, for each $\Psi_{v_g}: TM \to TM, w_m \mapsto T\Phi(v_g,w_m)$, the Vilms vertical projection operator has the property  $\vilms{P}_v \circ T(\Psi_{v_g})= T(\Psi_{v_g}) \circ 	\vilms{P}_v$. 

Remark first that a version of the Leibniz rule says that $T\Phi: TG\times TM \to TM$ satisfies
\[
\Psi_{v_g}(w_m) = T\Phi(v_g,w_m) = T(\Phi_g)(w_m +\tilde{\xi}(m)).
\]
One may find a version of this formula in e.g.\ Theorem~3.5 of \cite{Marrero} (if one takes $TM$ for the Lie algebroid $A$).
If we write $d_\xi: w_m \mapsto w_m +\tilde{\xi}(m)$, then $\Psi_{v_g} = T(\Phi_g)\circ d_\xi$. 
One easily verifies that for any map $F:M\to M$,   $\sigma \circ T(TF) = T(TF) \circ \sigma$ (see also \cite{natural}). With this, and with $T(\Phi_g) \circ  P_v=P_v \circ T(\Phi_g)$, we get
\[
\vilms{P}_v \circ T(\Psi_{v_g})= \sigma \circ TP_v \circ \sigma \circ T(T(\Phi_g)) \circ Td_\xi =  \sigma \circ T(P_v \circ  T(\Phi_g)) \circ \sigma \circ Td_\xi = T(T(\Phi_g))\circ \vilms{P}_v  \circ Td_\xi.
\] 
The result will follow if we can show that $\vilms{P}_v  \circ Td_\xi = Td_\xi \circ \vilms{P}_v$. In local coordinates, we may write
\[
d_\xi (x^i,y^\alpha,v^i,w^\alpha)= (x^i,y^\alpha,v^i,w^\alpha+\xi^\gamma K^\alpha_\gamma).
\]

On the one hand we have that $
(\vilms{P}_v  \circ Td_\xi) (x^i,y^\alpha,v^i,w^\alpha,X^i,Y^\alpha,V^i,W^\alpha)=$
\begin{eqnarray*} &&  \left(x^i,y^\alpha,v^i,w^\alpha+\xi^\gamma K^\alpha_\gamma,0,Y^\alpha-h^\alpha, 0, W^\alpha+\xi^\gamma\left(\fpd{K^\alpha_\gamma}{x^i}X^i  +  \fpd{K^\alpha_\gamma}{y^\beta}Y^\beta\right)\right.
\\ && \hspace*{8cm} \left.
 - \fpd{h^\alpha}{x^i}v^i-\fpd{h^\alpha}{y^\beta}(w^\beta+\xi^\gamma K^\beta_\gamma) -\fpd{h^\alpha}{v^i}V^i\right).
\end{eqnarray*}
On the other hand,
$
(Td_\xi \circ \vilms{P}_v) (x^i,y^\alpha,v^i,w^\alpha,X^i,Y^\alpha,V^i,W^\alpha)=
$
\[
  \left( x^i,y^\alpha,v^i,w^\alpha+\xi^\gamma K^\alpha_\gamma,0,Y^\alpha-h^\alpha, 0, W^\alpha   - \fpd{h^\alpha}{x^i}v^i-\fpd{h^\alpha}{y^\beta}w^\beta -\fpd{h^\alpha}{v^i}V^i  + \xi^\gamma\fpd{K^\alpha_\gamma}{y^\beta} (Y^\beta - h^\beta)    \right).
\]
The difference is, when written in full,
\[
\xi^\gamma\left(X^i \fpd{K^\alpha_\gamma}{x^i}(x,y)  + h^\beta(x,y,X)\fpd{K^\alpha_\gamma}{y^\beta}(x,y) -K^\beta_\gamma(x,y)\fpd{h^\alpha}{y^\beta}(x,y,X) \right).
\]
The factor between brackets vanishes, in view of the condition (\ref{help}) that expresses the invariance of the vector field $\Delta_h$.
\end{proof}

In this specific situation, we  may also define a principal nonlinear splitting on the $G$-principal bundle $TM\to (TM)/G$.
\begin{definition}
	The vertical lift splitting $\vlift{h}$ of a principal nonlinear splitting $h$ is the $G$-principal splitting on the principal bundle $\bar{\pi}:TM\rightarrow TM/G$, whose splitting map is $\tau^*\omega = \omega\circ T\tau$.
\end{definition}

The corresponding splitting is  principal because, for $X\in T_vTM$,  
\[
\omega(T\tau(T(T\Phi_g)X))=\omega(T\Phi_g T\tau(X))={\rm Ad}_{g}\omega(T\tau(X)).
\]
We will denote its corresponding vertical projection operator by $\vlift{P_v}: TTM \to TTM$.
	Let $\Gamma_0 = u^a  \partial/\partial q^a + F^a  \partial/\partial v^a$ be any \sode\ on $M$. Then the vector field
	\[
\Xi=  \vlift{P_v} (\Gamma_0) =  (K^{-1})_\gamma^\beta(w^{\gamma}-h^{\gamma})   \clift{\tilde E}_\beta =  (w^{\gamma}-h^{\gamma}) \left(  \frac{\partial}{\partial y^\delta} + (K^{-1})_\gamma^\beta  {\dot K}^\delta_\beta \frac{\partial}{\partial w^\delta}  \right)
\]
is clearly independent of the specific $F^a$, and therefore of the choice of $\Gamma_0$. Moreover it is tangent to ${\mathcal H}$ and it satisfies $TT\pi\circ \Xi=0$.

\begin{proposition}
	Let $h:M\rightarrow M/G$ be a principal splitting on the principal fibre bundle $\pi: M\rightarrow M/G$ and let $\bar{\Gamma}$ be a {\sode} on $M/G$. Then
	\[\label{unreduced}
		\Gamma = \vilms{\bar{\Gamma}}+\Xi
	\]
	is a {\sode} on $M$ which is tangent to ${\mathcal H}$. Furthermore, $\Gamma$ is submersive and it submerses to $\bar{\Gamma}$ through $\pi$.
\end{proposition}
\begin{proof} Assume that $\bar{\Gamma}$ has the local expression $\bar{\Gamma}={v}^i\fpd{}{x^i}+f^i(x,v)\fpd{}{v^i}$. Then, the Vilms-lift of $\bar{\Gamma}$ can, in vector field notation, be locally expressed as
	\[
	\vilms{\bar{\Gamma}} = v^i\fpd{}{x^i}+h^{\alpha}\fpd{}{y^\alpha}+f^i\fpd{}{v^i}+\left(\frac{\partial h^{\alpha}}{\partial x^j}v^j + \frac{\partial h^{\alpha}}{\partial y^{\beta}}w^{\beta} +\frac{\partial h^{\alpha}}{\partial v^j}f^j\right)\fpd{}{w^\alpha}.    
	\]
This vector field is tangent to ${\mathcal H}$ (the submanifold given by $w^\alpha-h^\alpha=0$). However, it is not a {\sode} on $M$, since we have terms $h^{\alpha}\frac{\partial}{\partial y^{\alpha}}$ instead of $w^{\alpha}\frac{\partial}{\partial y^{\alpha}}$. By adding $\Xi$ to $\vilms{\bar{\Gamma}}$, we get the \sode\
	\[
	\Gamma = v^i\fpd{}{x^i}+w^\alpha\fpd{}{y^\alpha}+f^i\fpd{}{v^i}+\left(\frac{\partial h^{\alpha}}{\partial x^j}v^j + \frac{\partial h^{\alpha}}{\partial y^{\beta}}w^{\beta} +\frac{\partial h^{\alpha}}{\partial v^j}f^j+(w^{\gamma}-h^{\gamma})(K^{-1})_\gamma^\beta  {\dot K}^\alpha_\beta\right)\fpd{}{w^\alpha}.
	\]
Since $\Xi$ is also tangent to ${\mathcal H}$, so is $\Gamma$. From the coordinate expression we see that $TT\pi\circ\Gamma =\bar\Gamma \circ T\pi$, which means that $\Gamma$ submerses to $\bar\Gamma$ through $\pi$. \end{proof}

A \sode\ $\Gamma$ can be characterized as a vector field on $TM$ that satisfies $S(\Gamma)=\Delta$, where $S= dq^a \otimes \partial/\partial u^a$ is the {\em vertical endomorphism} on $M$. It is remarkable that $\vilms{\bar{\Gamma}}$ and $\Xi$ are two vector fields on $M$ that satisfy
\[
S(\Xi)=\Delta_v \qquad\mbox{and}\qquad S(\vilms{\bar\Gamma}) = \Delta_h.
\]

\section{Curvature} \label{seccurvature}

In this section we define the curvature map of a nonlinear splitting, by analogy of the curvature tensor of an Ehresmann connection. 
\begin{definition}
	Let $h:\pi^*TN\rightarrow TM$ be a nonlinear splitting on $\pi:M\rightarrow N$. The curvature map is the operation $R:\vectorfields{M}\times \vectorfields{M} \rightarrow \vectorfields{M} $, given by
	\[
R(Z,W)=P_v\left[P_h(Z),P_h(W)\right].
	\]
\end{definition}
The above expression is clearly skew in $Z$ and $W$ and since $T\pi\circ P_v=0$ the result  is always a $\pi$-vertical vector field on $M$. For vector fields $X$ and $Y$ on $N$ we can compute that
\begin{eqnarray*}
R(X^h,Y^h)&=&  
P_v\left[X^h,Y^h\right] = \left[X^h,Y^h\right] - P_h\left[X^h,Y^h\right] \\&=& \left[X^h,Y^h\right] - P_h(\left[X, Y\right]^h) = \left[X^h,Y^h\right] - \left[X, Y\right]^h.  
\end{eqnarray*}
The first equality of the second line follows because  $T\pi\left( \left[X^h,Y^h\right]-[X,Y]^h\right) = [T\pi (X^h),T\pi(Y^h)] -[X,Y] = 0$. Therefore $\left[X^h,Y^h\right]=[X,Y]^h+V$ (for a $\pi$-vertical vector field $V$ on $M$) and thus $P_h\left[X^h,Y^h\right]=P_h([X,Y]^h) = [X,Y]^h$.  

For a $\pi$-vertical vector field $V$ on $M$, it also holds that $R(X^h+V,W) = R(X^h,W)$. For this reason, the curvature map is essentially determined by the association ${\bar R}: \vectorfields{N}\times \vectorfields{N} \to \vectorfields{M}$, given by 
\[
{\bar R}: (X,Y) \mapsto R(X^h,Y^h)=\left[X^h,Y^h\right] - \left[X, Y\right]^h.
\]

We remark that the curvature is not a tensor field. But, we can give an example where we can associate a linear map with it, even though $h$ is not an Ehresmann connection.

	Consider an affine map $h:\pi^*TN\to TM, (x^i,y^\alpha,v^i) \mapsto (x^i,y^\alpha,v^i,w^\alpha=-A^\alpha_iv^i + A^\alpha_0)$. In that case, its linear part, $H:\pi^*TN\to TM, (x^i,y^\alpha,v^i) \mapsto (x^i,y^\alpha,v^i,w^\alpha=-A^\alpha_iv^i)$ defines an Ehresmann connection. For a vector field $X=X^i \partial/\partial x^i$ on $N$, we will denote the horizontal lift by this Ehresmann connection by
	\[
	X^H=X^iH_i  \in \vectorfields{M}, \qquad \mbox{with\,\,\,} H_i=\fpd{}{x^i}-A^\alpha_i\fpd{}{y^\alpha}.
	\]
	Coming back to the affine splitting, its horizontal lift is
\[
	X^h=X^H + A_0.
	\]
With that, we may compute the curvature map as
\[
	{\bar R}(X,Y)= [X^i H_i,Y^j H_j]  + [X^iH_i,A_0] + [A_0,Y^jH_j]  -   [X,Y]^iH_i -A_0.
	\]
The first and the fourth term together represent the curvature of the Ehresmann  connection: If we denote $[H_i,H_j]=B^\alpha_{ij} \partial/\partial y^\alpha$, then these terms together give $X^iY^jB_{ij}^\alpha  \partial/\partial y^\alpha$.

If 
 we set $[H_j,A_0] =A^\alpha_{0j} \partial/\partial y^\alpha$, with 
\[
A^\alpha_{0j} = H_j(A_0^\alpha) + A_0(A^\alpha_j), 
\]
we may compute that\[
 {\bar R}(X,Y)= \left(X^iY^jB^\alpha_{ij} + X^iA^\alpha_{0i} - Y^jA^\alpha_{0j} + A_0^\alpha\right)\fpd{}{y^\alpha}.
\]
 From this expression, it is clear that, although the curvature map  is not tensorial,  in this case is exhibits an affine behaviour. We will use this to define a linear map, in two steps.

First, we may  define the value of the curvature map at a 'point', that is, give a meaning to ${\bar R}_{m}(u_n,v_n)$ with $u_n,v_n\in T_nN$, and $m\in M$ with $\pi(m)=n$:
	\[
{\bar R}_m(u_n,v_n):={\bar R}(X,Y)(m) \in T_mM,
	\]
	with $X,Y$ being arbitrary vector fields on $N$ satisfying $X(n)=u_n$ and $Y(n)=v_n$.
	
In the second step, we define a map ${\bar R}^0: \vectorfields{\bar\tau} \to \vectorfields{\tau}$, as follows. Let $\bar\zeta: TN \to TN$ be a vector field along $\bar\tau: TN \to N$. Then ${\bar R}^0(\zeta): TM \to TM$ is the vector field along $\tau$, defined by
\[
{\bar R}^0(\zeta)(w_m):={\bar R}_{m}(\zeta (T\pi(w_m)) ,T\pi(w_m)) - {\bar R}_{m}(0_n ,T\pi(w_m)).
\] 
With this, we get for $w_m=(x^i,y^\alpha,v^i,w^\alpha)$, $T\pi(w_m)=(x^i,v^i)$  and $\bar\zeta = {\bar\zeta}^i(x,v) \partial / \partial x^i$,
\[
{\bar R}^0(\zeta)= {\bar\zeta}^i\left(v^jB^\alpha_{ij}  +  A^\alpha_{0i} \right)\fpd{}{y^\alpha}.
	\]

We give two occurrences where the above tensorial object ${\bar R}^0$ makes its appearance.

	{\bf Example 1. Nonholonomic systems with affine constraints.} The equations of motion of a nonholonomic system are given by the Lagrange-d'Alembert equations. If $L\in \cinfty{TM}$ is the Lagrangian of the system, then (in the terminology of this paper) we may identify the `constrained Lagrangian' $L_c$ with the composition $L\circ P_h$,  where $P_h$ is the horizontal projector of the affine nonlinear splitting. It is shown in \cite{BKMM} that the Lagrange-d'Alembert equations can then be written as
	\[
	\left\{ \begin{array}{l} {\dot y}^\alpha + A^\alpha_i {\dot x}^i = A^\alpha_0,
 \\[2mm]\displaystyle\frac{d}{dt} \left( \fpd{L_c}{v^i} \right) - \fpd{L_c}{x^i} + A^\alpha_i \fpd{L_c}{y^\alpha} = \left(- B^\alpha_{ij} {\dot x}^j -A^\alpha_{0i} \right) \fpd{L}{w^\alpha}. \end{array} \right.
	\]
In the case where the constraints are linear (i.e.\ when $A^\alpha_0 =0$) the interpretation of the right-hand side is clear: it represents a force term that is determined by the curvature of the Ehresmann connection. From the above considerations it is clear that, also in the case of affine constraints, we can now interpret the right-hand side as a curvature: that of the affine nonlinear splitting.

{\bf Example 2. Magnetic Lagrangian systems.} We come back to the case of a principal bundle $\pi:M\rightarrow M/G$, but with an invariant Lagrangian of the following type
\begin{equation}\label{magnetic}
	L=T-V+A.
\end{equation}
Herein is $T$ the kinetic energy that one can associate with an invariant Riemannian metric. We make the further assumption that the vertical part of the metric (that
is, its restriction to the fibres of $\pi$) comes from a bi-invariant metric on $G$ (or, equivalently  an ${\rm Ad}$-invariant inner product on ${\mathfrak g}$). The potential $V$ is supposed to be an invariant function on $M$ (or: a function on $M/G$)  and the vector potential $A$ (representing magnetic forces) is the linear function that one can associate to an invariant 1-form on $M$.

In order to proceed we need to recall the notion of the mechanical (Ehresmann) connection.  Since the Hessian of $L$ w.r.t.\ fibre coordinates is positive definite, one can define the horizontal subspace of $TM$ as the orthogonal complement of the vertical space w.r.t.\ this Hessian. The corresponding principal connection on $\pi: M \to M/G$ is called the mechanical connection  (for further details, see \cite{Inv}).  

Let $x^i$ be coordinates on $M/G$ and let's now denote the horizontal lift w.r.t.\ the mechanical connection of the corresponding coordinate vector fields by $H_i$.  These are invariant vector fields on $M$.
If we also fix a basis $\{E_{\alpha}\}$ of $\mathfrak{g}$, we may consider  the invariant vector fields $\hat{E}_{\alpha}$ on $M$, generated by these elements, given  by
\[
{\hat E}_\alpha: (x,g) \mapsto \widetilde{(\Ad_{g} E_\alpha)}(x,g) =
T\Phi_g\big(\tilde{E}_\alpha(x,e)\big).
\]
Then, the set $\{H_i,\hat{E}_{\alpha}\}$ constitutes a frame field of $M$, consisting of only $G$-invariant vector fields. For later reference,  we write down the Lie brackets of these vector fields:
\[
[H_i,H_j] = K^\alpha_{ij}{\hat E}_\alpha, \qquad [H_i,{\hat E}_\alpha] = \Upsilon^\beta_{i\alpha}{\hat E}_\beta, \qquad [{\hat E}_\alpha,{\hat E}_\beta] = C^\gamma_{\alpha\beta}{\hat E}_\gamma
\]
Each of these brackets has a geometric interpretation: $K^\alpha_{ij}$ are the coefficients of the curvature of the mechanical connection, $C^\beta_{\alpha\gamma}$ are the structure constants of $\mathfrak{g}$  and  $\Upsilon^\beta_{i\alpha}$ are the coefficients of an adjoint  linear connection (see \cite{CMR}).

We will also   use their corresponding quasi-velocities $(\v^i,\w^{\alpha})$. More precisely, this means that, for any tangent vector $w_m$ in $T_mM$, $\v^i(m)$ and $\w^{\alpha}(m)$ are the components of $w_m$ with respect to the  basis $\{X_i(m),\hat{E}_{\alpha}(m)\}$ of $T_mM$. 
As a matter of fact $\v^i=v^i$.

The Lagrangian in (\ref{magnetic}) has then the following form:
\begin{equation}\label{magneticlagrangianinquasi}
	L=\frac{1}{2}g_{ij}v^iv^j+\frac{1}{2}k_{\alpha\beta}\w^{\alpha}\w^{\beta}-V+A_iv^i+A_{\alpha}\w^{\alpha},
\end{equation}
where, as a result of the assumed invariance conditions, $k_{\alpha\beta}$ are constants, satisfying 
\[
k_{\alpha\delta}C^\delta_{\beta\gamma} + k_{\beta\delta}C^\delta_{\alpha\gamma} = 0,
\]
and $g_{ij}$, $A_i$ and $A_\alpha$ are functions on $M$ that are independent of $y^{\alpha}$ (i.e.\ they are functions on $M/G$). For this reason, we can also interpret $L$ as the reduced Lagrangian $l$ on $(TM)/G$.

Since $g$ is assumed to be positive-definite, the  Lagrangian $L$ is fibre-regular. It therefore generates a non-linear splitting, in the sense of Definition~\ref{induced}. Since $\vlift{\hat E}_\alpha(L)=k_{\alpha\beta}\w^\beta + A_\alpha$, the components of the nonlinear splitting $h$, induced by (\ref{magneticlagrangianinquasi}), can be readily calculated in quasi-velocities to be
	\[
		\w^\alpha = {\sf h}^\alpha= -k^{\alpha\beta}A_{\beta},
	\]
		where $k^{\alpha\beta}$ denotes the matrix inverse of $k_{\alpha\beta}$. This means that the horizontal lift of the nonlinear splitting $h$ is in fact an affine map whose vector part can be related to the horizontal lift of the mechanical connection:
\[	
\left(X^i\fpd{}{x^i} \right)^h = X^iH_i-k^{\alpha\beta}A_{\beta} {\hat E}_\alpha= X^iH_i + A_0
	\]
We may again compute that
\begin{eqnarray*}
\nonumber {\bar R}(X,Y)&=& \left(X^iY^jK^\alpha_{ij} -X^ik^{\alpha\delta}\frac{\partial A_\delta}{\partial x^i}-X^ik^{\gamma\delta}A_\delta\Upsilon^\alpha_{i\gamma}\right.\\
\nonumber &&  \hspace*{5cm} \left.    +Y^jk^{\alpha\delta}\frac{\partial A_\delta}{\partial x^f}+Y^jk^{\gamma\delta}A_\delta\Upsilon^\alpha_{j\gamma} +k^{\alpha\delta}A_\delta \right)\hat{E}_\alpha.
\end{eqnarray*}
and\[
	{\bar R}_0(\zeta)=\zeta^i\left(v^jK^\alpha_{ij} -k^{\alpha\delta}\frac{\partial A_\delta}{\partial x^i}-k^{\gamma\delta}A_\delta\Upsilon^\alpha_{i\gamma}\right)\hat{E}_\alpha.
\]

We can relate the curvature map of the  induced nonlinear splitting to the submersiveness of the underlying {\sode}.
\begin{proposition}\label{cur}
	Let $L$ be an invariant Lagrangian  of the type (\ref{magnetic}) and let $\Gamma_L$ denote its Euler-Lagrange field.  If ${\bar R}_0$ vanishes,  $\Gamma_L$ is submersive through $\pi$ to an Euler-Lagrange \sode\ ${\bar\Gamma}_{\bar L}$ on $M/G$.
\end{proposition}

\begin{proof} Since $L$ is invariant its Euler-Lagrange equations can be reduced to the Lagrange-Poincar\'e equations of $l$. These can be written in  quasi-velocities as
\begin{eqnarray*}
&&\frac{d}{dt}\left( \fpd{l}{v^i}\right) -\fpd{l}{x^i} =
(-K^a_{ik}v^k + \Upsilon^a_{ib}\w^b)\fpd{l}{\w^a},\\
&&\frac{d}{dt}\left( \fpd{l}{\w^a}\right) = (\Upsilon^b_{ia}v^i +
C^b_{ac}\w^c)\fpd{l}{\w^b},
\end{eqnarray*}
(see e.g.  \cite{Inv}). In the case of a magnetic Lagrangian (\ref{magneticlagrangianinquasi}), we get
		\begin{eqnarray*}
		\frac{d}{dt}\left( g_{ij}v^j+A_i\right)-\frac{1}{2}\frac{\partial g_{jk}}{\partial x^i}v^jv^k +\fpd{V}{x^i}-\frac{\partial A_k}{\partial x^i}v^k -\frac{\partial A_\alpha}{\partial x^i}\w^{\alpha}&=&\left(-K^\alpha_{ij}v^j+\Upsilon^\beta_{i\alpha}\w^\beta\right) \left(k_{\beta\gamma}\w^\gamma +A_\alpha\right), \label{hor} \\
		\nonumber \frac{d}{dt}\left(k_{\alpha\gamma}\w^\gamma+A_\alpha\right)&=&\left(\Upsilon^\beta_{i\alpha}v^i+C^\beta_{\alpha\gamma}\w^\gamma\right)\left(k_{\beta\gamma}\w^\gamma +A_\beta\right).
	\end{eqnarray*}
One should interpret these equations as the coupled differential equations that determine an initial value problem in the unknown curve $(x^i(t),{\dot x}^i(t),\w^\alpha(t))$ of $(TM)/G$.	In case that the curvature ${\bar R}_0$ vanishes, however, we get that  
	\[
		-K^\alpha_{ij}v^jk_{\alpha\gamma}+\Upsilon^\alpha_{i\gamma}A_\alpha+\frac{\partial A_\gamma}{\partial x^i}=0.
	\]
Since moreover $k_{\beta\gamma}\Upsilon^\beta_{i\alpha}\w^\beta w^\gamma=0$ (see e.g.\ paragraph 6.1 in \cite{Inv}),  the first set of the Lagrange-Poincar\'e equations is independent of the variables $\w^{\alpha}$ and therefore constitute a subsystem on its own. This means that $\Gamma_L$ is submersive. The base integral curves of the corresponding \sode\ $\bar\Gamma$ on $M/G$ are solutions of the differential equations 
\[
		\frac{d}{dt}\left( g_{ij}v^j+A_i\right)-\frac{1}{2}\frac{\partial g_{jk}}{\partial x^i}v^jv^k+\fpd{V}{x^i}-\frac{\partial A_k}{\partial x^i}v^k =0.
	\]
These are the Euler-Lagrange equations of the Lagrangian $\bar L =  \frac{1}{2}g_{ij}v^iv^j -V+A_iv^i$ on $M/G$.
 \end{proof}

If we compare this situation to the more general one that we had discussed in Proposition~\ref{symmetryprop}, we see that now, regardless whether the curve is horizontal or not, we have the property that base integral curves of $\Gamma_L$ project to base integral curves of ${\bar\Gamma}_{\bar L}$.

\section{Outlook} \label{secoutlook}

In the current paper we have discussed fibre-regular Lagrangians $L$ and their nonlinear splittings. It is of interest to consider, roughly speaking, the special case where the Lagrangian is the energy of a Finsler function \cite{BCS2}. A function $F$ on $TM$ is a Finsler function when it is smooth on $\mathring{T}M$, positive, positive homogeneous and has the property that the Hessian of its energy function $E=\frac12 F^2$ with respect to fibre coordinates is a positive-definite matrix everywhere. As a consequence, any submatrix has non-vanishing determinant. Therefore, if $\pi:M \to N$ is a given fibre bundle, $E$ is always fibre regular. For this reason, we may consider the nonlinear splitting of $E$. 


\begin{proposition} If the Lagrangian function $L$ is a $2^+$-homogeneous fibre-regular function on $TM$, then the induced nonlinear splitting is homogeneous.
\end{proposition} 
\begin{proof} The homogeneity of the Lagrangian can be expressed as $\Delta(L)=2L$, where $\Delta$ is the Liouville vector field on $M$. Since, for each $X\in\vectorfields{M}$, $[\Delta,\vlift X] = -\vlift X$, it is easy to see that $\Delta(\vlift X(L)) = \vlift X(L)$, from which we may conclude that   $\vlift{X}(L)$ is a $1^+$-homogenous function on $TM$. As a consequence, for a $\pi$-vertical vector field $Y$ on $M$, both
\[
\vlift{Y}(L)(x,y,\lambda v,\lambda h(x,y,v)) =\lambda \vlift{Y}(L)(x,y,  v, h(x,y,v)) =0 \,\,\,\mbox{and}\,\,\, \vlift{Y}(L)(x,y,\lambda v, h(x,y,\lambda v))=0. 
\] 
The first item expresses the $1^+$-homogeneity of $\vlift{Y}(L)$, and the second the  definition of the induced nonlinear splitting $h$. Because of the uniqueness in the Implicit Function Theorem, we may conclude that $\lambda h^\alpha(x,y,v)= h^\alpha(x,y,\lambda v)$. \end{proof}

 The energy function $E=\frac12 F^2$ of a Finsler function $F$ is such a $2^+$-homogeneous regular Lagrangian. It is well-known that its Euler-Lagrange \sode\ $\Gamma_E$ is a spray. The Finsler function $F$ itself is a singular Lagrangian, and all `projectively related sprays' of the type $\Gamma_P=\Gamma_E -2 P \Delta$ (with $P$ a $1^+$-homogeneous function on $TM$) satisfy its Euler-Lagrange equations. The geometric interpretation of this property is that the base integral curves of $\Gamma_E$ are the geodesics of the Finsler function that are parametrized by arc length, while those of $\Gamma_P$ can be considered to be reparametrizations.   

Under the assumption that the further condition (\ref{symmetrycondition}) is satisfied, the vector field $\Gamma_E$ is tangent to ${\mathcal H}$ (Proposition~\ref{symmetryprop}). From Proposition~\ref{Ehresmann1} we know that the  homogeneous nonlinear splitting induced by $E$ has also the property that $\Delta$ is tangent to ${\mathcal H}$.  For this reason, also any projectively equivalent spray $\Gamma_P$ shares this property. Moreover, since $h$ is homogeneous, the subduced Lagrangian ${\bar E}$ of $E$ will be a $2^+$-homogeneous function. It would be of interest to know when it represents a Finsler metric. 

For homogeneous connections, we also know that there exist a horizontal lift of a reparametrized curve that remains a solution (Proposition~\ref{homprop}). From all this, we may conclude that the statement in Proposition~\ref{symmetryprop} is `geometric', in the sense that it does not depend on the specific chosen parametrization of a geodesic.
In a next contribution \cite{preprint}, we will investigate these aspects in more detail, both at the level of Finsler manifolds and Minkowski (vector) spaces. We will relate our results on subduced Finsler functions to the notion of a Finsler submersion, as it is called in \cite{Alv,libing}, and apply it to the case of Finsler geometry on homogeneous spaces.

{\bf Acknowledgements.} We are grateful to the referees for their valuable comments on the preprint version of this paper.

{\bf Data sharing statement.} Data sharing not applicable to this article as no datasets were generated or analysed during the current study.

\end{document}